\documentclass[12pt,letter]{amsart}

\usepackage{amsfonts}
\usepackage[letterpaper,margin=1in]{geometry}
\usepackage{amssymb,enumitem}
\usepackage{amsmath}
\usepackage{graphicx}
\usepackage[all]{xy}
\usepackage{amsthm}
\usepackage{color}
\usepackage{enumitem}

\usepackage[colorlinks, linkcolor=reference,
            citecolor=citation,
          urlcolor=e-mail]{hyperref}
\usepackage[numbers]{natbib}
\definecolor{e-mail}{rgb}{0,.40,.80}
\definecolor{reference}{rgb}{.20,.60,.22}
\definecolor{citation}{rgb}{0,.40,.80}
           
           \RequirePackage{verbatim}
\RequirePackage{eufrak}
\RequirePackage{microtype}
\RequirePackage{hyperref}

\RequirePackage{color}
\definecolor{todo}{rgb}{1,0,0}
\definecolor{answer}{rgb}{0,0,1}
\definecolor{new}{rgb}{1,0,1}
\definecolor{conditional}{rgb}{0,1,0}
\definecolor{e-mail}{rgb}{0,.40,.80}
\definecolor{reference}{rgb}{.20,.60,.22}
\definecolor{mrnumber}{rgb}{.80,.40,0}
\definecolor{citation}{rgb}{0,.40,.80}

\usepackage{times,pslatex}
\usepackage{bbm}

\newtheorem{proposition}{Proposition}[section]
\newtheorem{corollary}{Corollary}[section]
\newtheorem{theorem}{Theorem}[section]
\newtheorem{lemma}{Lemma}[section]

\theoremstyle{remark}

\theoremstyle{definition}

\newtheorem{ex}{Example}[section]

\newtheorem{de}{Definition}[section]

\DeclareMathOperator{\lc}{lc}
\DeclareMathOperator{\Char}{char}

\newcommand{\Z}{\mathbb{Z}}

\title{Commuting planar polynomial vector fields for conservative Newton systems}

\author{Joel Nagloo}
\address{CUNY Bronx Community College, Department of Mathematics and Computer Science, 2155 University Avenue
Bronx, NY 10453}
\email{joel.nagloo@bcc.cuny.edu}
\urladdr{http://fsw01.bcc.cuny.edu/joel.nagloo/}
\author{Alexey Ovchinnikov}
\address{CUNY Queens College, Department of Mathematics,
65-30 Kissena Blvd, Queens, NY 11367 and
CUNY Graduate Center, Ph.D. programs in Mathematics and Computer Science, 365 Fifth Avenue,
New York, NY 10016}
\email{aovchinnikov@qc.cuny.edu}
\urladdr{http://qc.edu/~aovchinnikov/}

\author{Peter Thompson}
\address{CUNY Graduate Center, Ph.D. program in Mathematics, 365 Fifth Avenue,
New York, NY 10016}
\email{pthompson@gradcenter.cuny.edu}
\urladdr{http://peterthompsonmath.wordpress.com}
\begin{document}

\begin{abstract}
We study the problem of characterizing polynomial vector fields that commute with a given polynomial vector field on a plane. It is a classical result that one can write down solution formulas for an ODE that corresponds to a planar vector field that possesses a linearly independent  commuting vector field. This problem is also central to the question of linearizability of vector fields. 
 Let $f \in K[x]$, where $K$ is a field of characteristic zero, and  $d$  the  derivation that corresponds to the differential equation $\ddot x = f(x)$ in a standard way. Let also $H$ be the Hamiltonian polynomial for $d$, that is $H=\frac{1}{2}y^2-\int{f(x)dx}$. It is known that the set of all polynomial derivations that commute with $d$ forms a $K[H]$-module $M_d$.
In this paper, we show that, for every such $d$, the module $M_d$ is of rank $1$ if and only if $\deg f\geqslant 2$. For example, the classical elliptic equation $\ddot x = 6x^2+a$, where $a \in \mathbb{C}$, falls into this category.
\end{abstract}
\maketitle

\section{Introduction}
We study the problem of characterizing polynomial vector fields that commute with a given polynomial vector field on a plane. It is a classical result that one can write down solution formulas for an ODE that corresponds to a planar vector field that possesses a linearly independent (transversal) commuting vector field  (see Theorem~\ref{thm:classicalint}). This problem is also central to the question of linearizability of vectors fields (cf. \cite{Grau06} and \cite{Sabatini97}). In what follows, we will use the standard correspondence between (polynomial) vector fields and derivations on (polynomial) rings. 
Let 
\begin{equation}\label{eq:cns}d = y\frac{\partial}{\partial x} +  f(x)\frac{\partial}{\partial y}\end{equation}
be a derivation, where $f$ is a polynomial with coefficients in a field $K$ of zero characteristic.  This derivation corresponds to a conservative Newton system, and so to the differential equation $\ddot x = f(x)$. %
Let $H$ be the Hamiltonian polynomial for $d$, that is $H=\frac{1}{2}y^2-\int{f(x)dx}$. Then the set of all polynomial derivations that commute with $d$ forms a $K[H]$-module $M_d$ \cite[Corollary 7.1.5]{nowicki94}. 
In this paper, we show that, for every such $d$, the module $M_d$ is of rank $1$ if and only if $\deg f\geqslant 2$. For example, the classical elliptic equation $\ddot x = 6x^2+a$, where $a \in \mathbb{C}$, falls into this category.

A characterization of commuting planar derivations in terms of a common Darboux polynomial is given by Petravchuk \cite{petravchuk09}. This was generalized to higher dimensions in \cite{LiDu2012} by Li and Du. In \cite{Guha13}, Choudhury and Guha used Darboux polynomials to find linearly independent commuting vector fields and to construct linearizations of the vector fields.
In the case in which $K$ is the real numbers, our result generalizes a result on conservative Newton systems with a center to the case in which a center may or may not be present.  A vector field has a center at point $P$ if there is a punctured neighborhood of $P$ in which every solution curve is a closed loop.  A center is called isochronous if every such loop has the same period.  It was proven by Villarini \cite[Theorem~4.5]{Villarini92} that, if $D_1$ and $D_2$ are commuting vector fields orthogonal at noncritical points, then any center of $D_1$ is isochronous.  The hypothesis of this result can be relaxed to the case in which $D_2$ is transversal to $D_1$ at noncritical points (cf. \cite[Theorem, p. 92]{Sabatini97}).  In light of this result, one approach to showing the nonexistence of a vector field commuting with $D$ is to show that $D$ has a non-isochronous center.  In fact, Amel'kin \cite[Theorem~11]{amelkin77} has shown that if the system of ordinary differential equations (ODEs) corresponding to derivation~\eqref{eq:cns} is not linear and has a center at the origin, then there is no transversal vector field that commutes with $d$.  

As far as we are aware, there has not been a standard method to show the nonexistence of a transversal polynomial vector field in the absence of a nonisochronous center.  We develop our own method to do this, which includes building a triangular system of differential equations.  One technique we use in approaching this system involves constructing a family of pairs of commuting derivations on rings of the form $K[x^{1/t},x^{-1/t},y]$ (see Lemma~\ref{lem:positive}) and using recurrence relations.

It is impossible to remove the condition  $\deg f\geqslant 2$ from the statement of our main result, as every non-zero derivation of degree less than $2$ commutes with another transversal derivation (see Proposition~\ref{prop:linear}).  The form of $d$ in our main result implies that $d$ is divergence free (which is the same as Hamiltonian in the planar case).
It is not possible to strengthen our result to the case in which $d$ is merely assumed to be divergence free of degree at least $2$, as shown in Example~\ref{ex:nonlinear} and Proposition~\ref{prop:divfree}.

The paper is organized as follows. We introduce the basic terminology in Section~\ref{sec:defs}. The main result, Theorem~\ref{thm:main}, is stated and proven in Section~\ref{sec:main}.
\section{Basic terminology and related results}\label{sec:defs}
We direct the reader to \cite{Kap,Kol} for the basics of a ring with a derivation.
\begin{de}
An $S$-derivation on a commutative ring $R$ with subring $S$ is a map $d\colon R \to R$ such that $d(S) = 0$ and for all $a$, $b \in R$,
$$
d(a+b)=d(a)+d(b)\quad\text{and}\quad d(ab)=d(a)\cdot b+a\cdot d(b).
$$
\end{de}
\begin{de}Let $K$ be a field.  A non-zero $K$-derivation $d$ on $K[x_1,\ldots,x_n]$ is called {\em integrable} if there exist commuting $K$-derivations $\delta_1,\ldots,\delta_{n-1}$ on $K[x_1,\ldots,x_n]$ 
that are linearly independent from $d$ over $K(x_1,\ldots,x_n)$, and commute with $d$,  that is, for all $a \in K[x_1,\ldots,x_n]$ and $i$,$j$, $1\leqslant i,j\leqslant n-1$, 
$$d(\delta_i(a)) = \delta_i(d(a))\quad\text{and}\quad \delta_i(\delta_j(a))=\delta_j(\delta_i(a)).$$ 
\end{de}
The following is a classical result.
\begin{theorem}\label{thm:classicalint}
Let $d$ and $\delta$ be $\mathbb{R}$-derivations on $\mathbb{R}(x,y)$ defined by \[d(x) = f_1(x,y),\ d(y) = f_2(x,y),\ \delta(x) = g_1(x,y),\ \delta(y) = g_2(x,y).\]  Let $(x_0,y_0) \in \mathbb{R}^2$.  Suppose that $d$ and $\delta$ commute and there is no $(\lambda_1,\lambda_2)\in\mathbb{R}^2\backslash \{(0,0)\}$ such that \[\lambda_1\begin{pmatrix}f_1(x_0,y_0) \\ f_2(x_0,y_0)\end{pmatrix} = \lambda_2\begin{pmatrix}g_1(x_0,y_0) \\ g_2(x_0,y_0)\end{pmatrix}.\]  Then the initial value problem \[\dot{x} = f_1(x,y),\ \dot{y} = f_2(x,y),\ x(0) = x_0,\ y(0) = y_0\] has a solution given by \[(x(t),y(t)) = F^{-1}(t,0),\] where 
\[F\begin{pmatrix}x \\ y\end{pmatrix} = 
\begin{pmatrix}\displaystyle\int\limits_{x_0}^x\tfrac{g_2(r,y)}{\Delta(r,y)}dr + \int\limits_{y_0}^y\tfrac{-g_1(x_0,s)}{\Delta(x_0,s)}ds \\ 
\displaystyle\int\limits_{x_0}^x\tfrac{-f_2(r,y)}{\Delta(r,y)}dr + \int\limits_{y_0}^y\tfrac{f_1(x_0,s)}{\Delta(x_0,s)}ds \end{pmatrix}, \]
and $\Delta(x,y) = f_1(x,y)g_2(x,y)-f_2(x,y)g_1(x,y)$.
\end{theorem}
\begin{ex}\label{ex:nonlinear}
Consider the initial value problem \[\dot{x} = 1+x^2, \hspace{5mm} \dot{y} = -2xy, \hspace{5mm} x(0) = x_0, \hspace{5mm} y(0) = y_0,\] where $x_0$ and $y_0$ are real numbers and $y_0 \neq 0$.  The corresponding derivation is \[d(x) = 1+x^2, \hspace{5mm} d(y) = -2xy,\] and we observe that the derivation \[\delta(x) = 0, \hspace{5mm} \delta(y) = y\] commutes with $d$, and that $d$ and $\delta$ are independent at $(x_0,y_0)$.  Using the above formula, we obtain the solution \[x(t) = \tan(t+\tan^{-1}x_0), \hspace{5mm} y(t) = y_0(1+x_0^2)\cos^2(t+\tan^{-1}x_0).\]
\end{ex}

We make some observations, in the form of the following propositions:

\begin{proposition}\label{prop:linear}Let $K$ be a field.  Every non-zero $K$-derivation of degree less than or equal to $1$ on $K[x,y]$ is integrable.\end{proposition}

\begin{proof}
We will consider the following cases.  The symbols $a$, $b$, $c$, $e$, $f$, and $g$ are taken to be elements of $K$.
\begin{enumerate}[leftmargin=0.52in,label=Case~\arabic*]\addtocounter{enumi}{-1}
\item\label{case:0}:  $d(x) = c, \quad d(y) = g$. 
Observe that $d$ commutes with any constant derivation.
\item\label{case:1}: $d(x) = ax, \quad d(y) = ay, \quad a\neq 0$. 
Observe that $d$ commutes with $\delta$, where $\delta(x) = y$, $\delta(y) = x$.
\item\label{case:2}:  $d(x) = ax+by, \quad d(y) = ex+fy$, \quad different from~\ref{case:1}. 
Observe that $d$ commutes with $\delta$, where $\delta(x) = x, \quad \delta(y) = y$.
\item:  $d(x) = ax+by+c, \quad d(y) = ex+fy+g, \quad af-be \neq 0$. 
In this case, $d$ is equivalent to a derivation from \ref{case:1} or \ref{case:2} via a linear change of coordinates.  Let $(x_0,y_0)$ be the solution to the system $ax+by+c=ex+fy+g=0$.  Now let $u = x-x_0$ and $v = y-y_0$, so that $d(u) = au+bv$ and $d(v) = eu+fv$.
\item:  $d(x) = ax+by+c, \quad d(y) = ex+fy+g, \quad af-be = 0$
\begin{enumerate}
\item\label{case:4a} $a = b = 0$, different from \ref{case:0}. 
If $e\neq 0$, then $d$ commutes with and is transversal to $\delta$ given by $\delta(x) = -\frac{g}{e}$, $\delta(y) = 0$.  If $f \neq 0$, then $d$ commutes with and is transversal to $\delta$ given by $\delta(x) = 0$, $\delta(y) = -\frac{g}{f}$.
\item at least one of $a$ and $b$ is not $0$. 
First assume $a\neq 0$.  If $f=e=0$, then this is equivalent to~\ref{case:4a} by swapping the roles of $x$ and $y$.  Assume at least one of $f$ and $e$ is not $0$.  By the condition $af-be = 0$, it must be that $e\neq 0$.  Using the coordinate $z = ex - ay$ instead of $x$ puts this into the form of~\ref{case:4a}.  Next, assume $b \neq 0$.  If $f=e=0$, then this is equivalent to ~\ref{case:4a}.  Assume at least one of $f$ and $e$ is not $0$.  By the condition $af-be=0$, it must be that $f \neq 0$.  Using the coordinate $z = fx-by$ instead of $x$ puts this into the form of ~\ref{case:4a}.\qedhere
\end{enumerate}
\end{enumerate}
\end{proof}

\begin{de}
Let $K$ be a field and let $d$ be a $K$-derivation on $K[x_1,\ldots,x_n]$.  We say $d$ is {\em divergence-free} if 
$$
\sum_{i=1}^n \frac{\partial}{\partial x_i} d(x_i) = 0.
$$
\end{de}

\begin{proposition}\label{prop:divfree}Let $K$ be a field of characteristic $0$.  There exist integrable divergence-free $K$-derivations on $K[x,y]$ that are not coordinate-change equivalent to a derivation of degree less than or equal to $1$.\end{proposition}
\begin{proof}
The $K$-derivation defined by the same equations as $d$ from Example~\ref{ex:nonlinear} is divergence-free and integrable.  Note that the vector field corresponding to $d$ vanishes only at the points $(\sqrt{-1},0)$ and $(-\sqrt{-1},0)$  in $\overline{K}^2$.  Since $\text{char} K = 0$, these points are distinct.  After a coordinate change, the number of points in $\overline{K}^2$ at which a vector field vanishes does not change.  The vector field of any derivation of degree less than or equal to $1$ vanishes at zero, one, or infinitely many points.  We conclude that $d$ is not coordinate-change equivalent to a derivation  of degree no greater than $1$.
\end{proof}

In the following section, we study a class of divergence-free vector fields.  We show that no member of this class is integrable.

\section{Main result}\label{sec:main}

Fix a field $K$ of characteristic $0$.  Suppose $\delta_f$ represents a second-order differential equation of the form $$\ddot{x} = f,$$ where $f \in K[x]\setminus K$, which corresponds to a conservative Newton system.  That is,
\begin{equation}\label{eq:deltaf}
\delta_f\begin{pmatrix}x \\ y\end{pmatrix} = \begin{pmatrix}y \\ f\end{pmatrix}
\end{equation} 
If $\deg f = 1$, then $\delta_f$ is integrable by Proposition~\ref{prop:linear}.  The following theorem, which is our main result, addresses the case of $\deg f \geqslant 2$.

\begin{theorem}\label{thm:main} For every
\begin{itemize}
\item
$f\in K[x]$ such that $\deg f \geqslant 2$ and
\item $K$-derivation $\gamma$ on $K[x,y]$ that commutes with $\delta_f$, where
$\delta_f$  the $K$-derivation defined by ~\eqref{eq:deltaf},
\end{itemize}
 there exists $q \in K[H]$ such that \[\gamma = q\cdot\delta_f,\] where $H = y^2 - 2\int fdx$ and $\int fdx$ has $0$ as the constant term.
\end{theorem}

As a corollary, we recover the following result on conservative Newton systems with a center at the origin.  This result was first proven in \cite[Theorem~11]{amelkin77}  and was given new proofs in \cite[Theorem~4.1]{chicone89} and \cite[Corollary 2.6]{cima99} (see also \cite[p.~30]{volokitin99}).

\begin{corollary}
The real system
\begin{align*}
\dot{x} &= -y\\
\dot{y} &= f(x),
\end{align*}
with $f(0) = 0$, $f'(0) = 1$,
has a transversal commuting polynomial derivation if and only if $f(x) = x$.
\end{corollary}

\begin{proof}[Proof of Theorem~\ref{thm:main}]
Fix $f \in K[x]$ such that $\deg f \geqslant 2$.  Fix a $K$-derivation $\delta$ so that $\delta(x) = y$ and $\delta(y) = f$.  Fix a $K$-derivation $\gamma$ such that $[\delta,\gamma] = 0$.  First consider the case in which $\deg_y \gamma \leqslant 1$.

\begin{lemma}\label{lem:lowdeg}
If $$\gamma\begin{pmatrix}x \\ y\end{pmatrix} = \begin{pmatrix}c_1y+c_0 \\ d_1y+d_0\end{pmatrix},$$ where $c_1,c_0,d_1,d_0\in K[x]$, and $[\delta,\gamma] = 0$, then $$\gamma\begin{pmatrix}x \\ y\end{pmatrix} = c_1\delta.$$
\end{lemma}
\begin{proof}
The equations $\delta(\gamma(x)) = \gamma(\delta(x))$ and $\delta(\gamma(y)) = \gamma(\delta(y))$ yield
$$
\begin{cases}
c_1'y^2+c_0'y + fc_1 = d_1y+d_0\\
d_1'y^2 + d_0'y + fd_1 = f'c_1y + f'c_0.
\end{cases}
$$
Equating coefficients of like powers of $y$, we obtain the two independent systems
\begin{equation}\label{eq:sys1}
c_1' = 0, \quad d_0' = c_1f', \quad fc_1 = d_0
\end{equation}
and
\begin{equation}\label{eq:sys2}d_1' = 0, \quad c_0'=d_1, \quad fd_1 = c_0f'.
\end{equation}
The solution set of~\eqref{eq:sys1} is $c_1 = \text{constant}, d_0 = c_1f$.   System~\eqref{eq:sys2} has no non-zero solution, which we deduce as follows.  We have $$\left(\frac{c_0}{f}\right)' = \frac{c_0'f -f'c_0}{f^2} = 0,$$ so $c_0 = (\text{const})f$.  Therefore, $d_1 = (\text{const})f'$, which implies $d_1'=(\text{const})f'' = 0$.  Since we assume $\deg f \geqslant 2$, the constant must be $0$.  Therefore, \[\gamma\begin{pmatrix}x \\ y\end{pmatrix} = c_1\begin{pmatrix}y \\ f\end{pmatrix}.\qedhere\]
\end{proof}

Now assume $\deg_y \gamma = M\geqslant 2$.  Write 
\begin{equation}\label{g as a polynomial}
\gamma\begin{pmatrix}x \\ y\end{pmatrix} = \begin{pmatrix}c_My^M +\ldots + c_0 \\ d_My^M + \ldots + d_0\end{pmatrix},
\end{equation}
where for all $i$, $c_i, d_i \in K[x]$.  Since $M=\deg_y\gamma$, at least one of $c_M$ and $d_M$ is non-zero.  Now the system 
$$\begin{pmatrix}\delta(\gamma(x)) \\ \delta(\gamma(y))\end{pmatrix} = \begin{pmatrix}\gamma(\delta(x)) \\ \gamma(\delta(y))\end{pmatrix}$$  becomes 
\begin{multline}\label{expandDiffEq}
\begin{pmatrix}c_M'y^{M+1}+c_{M-1}'y^M+\ldots+c_0'y \\ d_M'y^{M+1}+d_{M-1}'y^M+\ldots+d_0'y\end{pmatrix}+\begin{pmatrix}Mfc_My^{M-1}+\ldots+fc_1 \\ Mfd_My^{M-1}+\ldots+fd_1\end{pmatrix}\\ =\begin{pmatrix}0 & 1 \\ f' & 0\end{pmatrix}\begin{pmatrix}c_My^M +\ldots + c_0 \\ d_My^M + \ldots + d_0\end{pmatrix}.
\end{multline}
Viewing these matrix entries as polynomials in $y$ and equating coefficients yields the following system of first-order ODEs
\begin{center}
\begin{tabular}{cc}
$c_M' = 0$ & $d_M' = 0$ \\
$c_{M-1}' = d_M$ & $d_{M-1}' = f'c_M$ \\
$c_{M-2}' + Mfc_M = d_{M-1}$ & $d_{M-2}' + Mfd_M = f'c_{M-1}$ \\
$c_{M-3}' + (M-1)fc_{M-1} = d_{M-2}$ & $d_{M-3}' + (M-1)fd_{M-1} = f'c_{M-2}$ \\
$c_{M-4}' + (M-2)fc_{M-2} = d_{M-3}$ & $d_{M-4}' + (M-2)fd_{M-2} = f'c_{M-3}$ \\
$c_{M-5}' + (M-3)fc_{M-3} = d_{M-4}$ & $d_{M-5}' + (M-3)fd_{M-3} = f'c_{M-4}$ \\
$\vdots$ & $\vdots$ \\
$c_0' + 2fc_2 = d_1$ & $d_0' + 2fd_2 = f'c_1$ \\
$fc_1 = d_0$ & $fd_1 = f'c_0$
\end{tabular}
\end{center}
as well as the condition \[c_M\neq 0 \text{ or } d_M \neq 0.\]
In each equation, it is the case that if $c_i$ and $d_j$ both appear, then $i$ and $j$ have opposite parities.  Thus, we see that this system consists of two independent systems.  If $M$ is odd, these systems are:
\begin{center}
\begin{tabular}{cc}
$(Io)_M$ & $(IIo)_M$ \\
$c_M' = 0$ & $d_M' = 0$ \\
$d_{M-1}' = f'c_M$ & $c_{M-1}' = d_M$ \\
$c_{M-2}' + Mfc_M = d_{M-1}$ & $d_{M-2}' + Mfd_M = f'c_{M-1}$ \\
$d_{M-3}' + (M-1)fd_{M-1} = f'c_{M-2}$ & $c_{M-3}' + (M-1)fc_{M-1} = d_{M-2}$ \\
$c_{M-4}' + (M-2)fc_{M-2} = d_{M-3}$ & $d_{M-4}' + (M-2)fd_{M-2} = f'c_{M-3}$ \\
$d_{M-5}' + (M-3)fd_{M-3} = f'c_{M-4}$ & $c_{M-5}' + (M-3)fc_{M-3} = d_{M-4}$\\
$\vdots$ & $\vdots$ \\
$c_1' + 3fc_3 = d_2$ & $d_1' + 3fd_3 = f'c_2$ \\
$d_0' + 2fd_2 = f'c_1$ & $c_0' + 2fc_2 = d_1$\\
$fc_1 = d_0$ & $fd_1 = f'c_0$
\end{tabular}
\end{center}
If $M$ is even, the systems are:
\begin{center}
\begin{tabular}{ccc}
$(IIe)_M$ & $(Ie)_M$ \\
$c_M' = 0$ & $d_M' = 0$ \\
$d_{M-1}' = f'c_M$ & $c_{M-1}' = d_M$ \\
$c_{M-2}' + Mfc_M = d_{M-1}$ & $d_{M-2}' + Mfd_M = f'c_{M-1}$ \\
$d_{M-3}' + (M-1)fd_{M-1} = f'c_{M-2}$ & $c_{M-3}' + (M-1)fc_{M-1} = d_{M-2}$ \\
$c_{M-4}' + (M-2)fc_{M-2} = d_{M-3}$ & $d_{M-4}' + (M-2)fd_{M-2} = f'c_{M-3}$ \\
$d_{M-5}' + (M-3)fd_{M-3} = f'c_{M-4}$ & $c_{M-5}' + (M-3)fc_{M-3} = d_{M-4}$\\
$\vdots$ & $\vdots$ \\
$c_0' + 2fc_2 = d_1$ & $d_0' + 2fd_2 = f'c_1$ \\
$fd_1 = f'c_0$ & $fc_1 = d_0$
\end{tabular}
\end{center}
In light of these observations, let
\begin{gather*}
n = \max\{i \:|\: i \text{ odd and } c_i\neq 0 \text{ or } i \text{ even and } d_i \neq 0\},\\
p = \max\{i \:|\: i \text{ even and } c_i\neq 0 \text{ or } i \text{ odd and } d_i \neq 0\}.
\end{gather*}
Note that $n$ or $p$ may be undefined.  Now write $\gamma = \gamma_1 + \gamma_2$, where $\gamma_1(x)$ contains the terms of $\gamma(x)$ of odd degree in $y$, $\gamma_1(y)$ contains the terms of $\gamma(y)$ of even degree in $y$, $\gamma_2(x)$ contains the terms of $\gamma(x)$ of even degree in $y$, and $\gamma_2(y)$ contains the terms of $\gamma(y)$ of odd degree in $y$.  Explicitly,
$$
\gamma_1\begin{pmatrix}x \\ y\end{pmatrix} = 
\begin{cases}
\begin{pmatrix}c_ny^n+c_{n-2}y^{n-2}+\ldots+c_1y \\ d_{n-1}y^{n-1}+d_{n-3}y^{n-3}+\ldots+d_0\end{pmatrix} & \text{ if } n \text{ is odd},
\\ \\ \begin{pmatrix}c_{n-1}y^{n-1}+c_{n-3}y^{n-3}+\ldots+c_1y \\ d_ny^n+d_{n-2}y^{n-2}+\ldots+d_0\end{pmatrix} &\text{ if } n \text{ is even},
\\ \\ \begin{pmatrix}0 \\ 0\end{pmatrix} & \text{ if } n \text{ is undefined,}
\end{cases}
$$
and
$$
\gamma_2\begin{pmatrix}x \\ y\end{pmatrix} = 
\begin{cases}
\begin{pmatrix}c_{p-1}y^{p-1}+c_{p-3}y^{p-3}+\ldots+c_0 \\ d_py^p+d_{p-2}y^{p-2}+\ldots+d_1y\end{pmatrix} & \text{if } p \text{ is odd},
\\ \\ \begin{pmatrix}c_py^p+c_{p-2}y^{p-2}+\ldots+c_0 \\ d_{p-1}y^{p-1}+d_{p-3}y^{p-3}+\ldots+d_1y\end{pmatrix} &\text{if } p \text{ is even,}
\\ \\ \begin{pmatrix}0 \\ 0\end{pmatrix} & \text{ if } p \text{ is undefined.}
\end{cases}
$$
As we have seen, the criterion $[\delta,\gamma ]=0$ is equivalent to the conjunction of two systems of equations in which one system only involves the terms of $\gamma_1$ and the other only involves the terms of $\gamma_2$.  Hence, $[\delta,\gamma_1] = [\delta,\gamma_2] = 0$.

Let us examine the possible values of $n$.  If $n$ is undefined, then $\gamma_1(x,y) = (0,0)$.  If $n = 0$, then $\gamma_1$ is the same as the $\gamma$ of Lemma~\ref{lem:lowdeg} with $c_1 = c_0 = d_1 = 0$.  Thus, by Lemma~\ref{lem:lowdeg}, $\gamma_1 = 0$, which contradicts that $n=0$.  If $n = 1$, then $\gamma_1$ is the same as the $\gamma$ of Lemma~\ref{lem:lowdeg} with $c_0 = d_1 = 0$.  Thus by Lemma~\ref{lem:lowdeg}, $\gamma_1 = c_1\delta$, and, in the proof of Lemma~\ref{lem:lowdeg}, it is shown that $c_1 \in K$.  If $n \geqslant 2$ is even, the coefficients of $\gamma_1$ must satisfy $(Ie)_n$ and $d_n \neq 0$.  We will show in Lemma~\ref{lem:Ie implies d_m=0} and Corollary~\ref{cor:nnoteven} that this is impossible.  If $n$ is odd, the coefficients of $\gamma_1$ must satisfy $(Io)_n$ and $c_n \neq 0$.  We will show in Lemma~\ref{lem:01} and Lemma~\ref{lem:02} that this implies $\gamma_1 = q\delta$ for some $q \in K[H]$.   In summary,
\begin{itemize}

\item If $n$ is undefined, then $\gamma_1 = 0 \cdot \delta$.

\item It is impossible that $n=0$.

\item If $n=1$, then $\gamma_1 = c_1 \cdot \delta$ and $c_1 \in K$.

\item It is impossible that $n \geqslant 2$ is even. (Lemma \ref{lem:Ie implies d_m=0}, Corollary \ref{cor:nnoteven})

\item If $n \geqslant 3$ is odd, then $\gamma_1 = q \cdot \delta$ for some $q \in K[H]$. (Lemmas \ref{lem:01}, \ref{lem:02})

\end{itemize}

Let us examine the possible values of $p$.  If $p$ is undefined, then $\gamma_2(x,y) = (0,0)$.  If $p = 0$, then $\gamma_2$ is the same as the $\gamma$ from Lemma~\ref{lem:lowdeg} with $c_1 = d_1 = d_0 = 0$.  Thus, by Lemma~\ref{lem:lowdeg}, $\gamma_2 = 0$, which contradicts that $p=0$.  If $p = 1$, then $\gamma_2$ is the same as the $\gamma$ of Lemma~\ref{lem:lowdeg} with $c_1 = d_0 = 0$.  Thus, by Lemma~\ref{lem:lowdeg}, $\gamma_2 = 0$, which contradicts that $p=1$.  If $p \geqslant 2$ is even, the coefficients of $\gamma_2$ must satisfy $(IIe)_p$ and $c_p \neq 0$.  We will show in Lemma~\ref{lem:IIe implies c_m=0} and Corollary~\ref{cor:pnoteven} that this is impossible.  If $p \geqslant 3$ is odd, the coefficients of $\gamma_2$ must satisfy $(IIo)_p$ and $d_p \neq 0$.  We will show in Lemma~\ref{lem:polynomial}, Lemma~\ref{lem:positive}, Lemma~\ref{lem:07}, and Corollary~\ref{cor:pnotodd} that this is impossible.  We summarize these results as follows:
\begin{itemize}

\item If $p$ is undefined, then $\gamma_2 = 0 \cdot \delta$.

\item It is impossible that $p=0$.

\item It is impossible that $p=1$.

\item It is impossible that $p \geqslant 2$ is even. (Lemma \ref{lem:IIe implies c_m=0}, Corollary \ref{cor:pnoteven})

\item It is impossible that $p \geqslant 3$ is odd. 
(Lemmas \ref{lem:polynomial}, \ref{lem:positive}, \ref{lem:07}, Corollary \ref{cor:pnotodd})

\end{itemize}

From the bulleted statements, it follows that $\gamma_1 = q\delta$ for some $q \in K[H]$ and $\gamma_2 = 0$.
These lemmas and their corollaries constitute the rest of the proof of Theorem~\ref{thm:main}.

\vspace{3mm}

\begin{de}Let $a \in K[x,y]$.  We define $\int a dx$ to be the element of $K[x,y]$ whose partial derivative with respect to $x$ is $a$ and whose constant term is $0$.
\end{de}

\begin{lemma}\label{lem:01}
For every odd integer $m \geqslant 3$,  the solution set of $(Io)_m$, with $c_0,\ldots,d_m$ treated as variables, is an $\frac{m+1}{2}$-dimensional $K$-vector space.
\end{lemma}

\begin{proof}
Fix $m \geqslant 3$.  Label the equations of $(Io)_m$ as follows:
\begin{center}
\begin{tabular}{cc}
$e_{m+1}$ & $c_m' = 0$ \\
$e_m$ & $d_{m-1}' = f'c_m$ \\
$e_{m-1}$ & $c_{m-2}' + mfc_m = d_{m-1}$ \\
$\vdots$ & $\vdots$ \\
$e_1$ & $d_0' + 2fd_2 = f'c_1$ \\
$e_0$ & $fc_1 = d_0$
\end{tabular}
\end{center}
We show the following by induction on $k$, $0 \leqslant k \leqslant \frac{m-3}{2}$:

\begin{equation}\label{eqn:indstatement}
\text{\parbox{.85\textwidth}{The solution set of $\{e_{m+1},e_m,\ldots,e_{m-2k-2},d_{m-2k-3}=fc_{m-2k-2}\}$ is a $K$-vector space of dimension $k+2$.}}
\end{equation}

\vspace{3mm}

\noindent \underline{Base Case: $k=0$}

The system \begin{equation}\label{eq:em1fcm2}\{e_{m+1},e_m,e_{m-1},e_{m-2},d_{m-3}=fc_{m-2}\}
\end{equation} is
\begin{center}
\begin{tabular}{ll}
$e_{m+1}:\qquad$ & $c_m' = 0$ \\ $e_m:$ & $d_{m-1}' = f'c_m$ \\ $e_{m-1}:$ & $c_{m-2}' = -mfc_m + d_{m-1}$ \\ $e_{m-2}:$ & $d_{m-3}' = -(m-1)fd_{m-1} + f'c_{m-2}$ \\ & $d_{m-3} = fc_{m-2}$
\end{tabular}
\end{center}
 Let $\big(\tilde{d}_{m-3},\tilde{c}_{m-2},\tilde{d}_{m-1},\tilde{c}_m\big)$ be a solution of~\eqref{eq:em1fcm2}.   By $e_{m+1}$, $\tilde{c}_m = a_1$ for some $a_1 \in K$.  
It follows that
$$
f'\tilde{c}_{m-2}+f\tilde{c}_{m-2}' = -(m-1)f\tilde{d}_{m-1} + f'\tilde{c}_{m-2},
$$
and hence
$$
\tilde{c}_{m-2}' = -(m-1)\tilde{d}_{m-1},
$$
and so
\begin{align*}
\tilde{d}_{m-1} &= mf\tilde{c}_m + \tilde{c}_{m-2}' = mf\tilde{c}_m - (m-1)\tilde{d}_{m-1}.
\end{align*}
Thus \[\tilde{d}_{m-1} = f\tilde{c}_m = a_1f.\]
It follows from this and $e_{m-1}$ that \[\tilde{c}_{m-2}' = -(m-1)a_1f,\] and hence
\[
\tilde{c}_{m-2} = -\int (m-1)a_1f dx + a_2
\]
for some $a_2 \in K$.  From this and the condition $\tilde{d}_{m-3} = f\tilde{c}_{m-2}$ it follows that
$$
\tilde{d}_{m-3} = f\left(-\int (m-1)a_1f dx + a_2\right).
$$
One can verify that
\begin{equation}\label{eqn:soltuple}
\Big(f\big(-\int (m-1)a_1f dx + a_2\big), \quad -\int (m-1)a_1f dx + a_2, \quad a_1f, \quad a_1\Big)
\end{equation}
 is indeed a solution of~\eqref{eq:em1fcm2}. 
We have just shown that the solution set of~\eqref{eq:em1fcm2} is exactly the elements of $K[x]^4$ of the form \eqref{eqn:soltuple} with $a_1, a_2 \in K$.  This set is the $K$-span of the tuples 
\[
\Big(f\big(-\int (m-1)f dx \big), \ -\int (m-1)f dx, \  f, \ 1\Big)
\quad
\text{and}
\quad
(f, \ 1, \ 0, \ 0).\]
Hence, the solution space is a two-dimensional $K$-vector space.

\vspace{3mm}

\noindent \underline{Inductive Step:}  Fix $k$, $0 \leqslant k < \frac{m-3}{2}$.  
Consider
\begin{gather}
\{e_{m+1},e_m,\ldots,e_{m-2k-2},d_{m-2k-3}=fc_{m-2k-2}\}\label{stepk}\\
\{e_{m+1},e_m,\ldots,e_{m-2k-4},d_{m-2k-5}=fc_{m-2k-4}\}\label{stepkplus1}
\end{gather}
Assume
\begin{equation}\label{eqn:ih}
\text{\parbox{.9\textwidth}{The solution set of~\eqref{stepk} 
is a $K$-vector space of dimension $k+2$.}}
\end{equation}
We will show
\begin{equation}\label{eqn:indstep}
\text{\parbox{.9\textwidth}{The solution set of~\eqref{stepkplus1} 
is a $K$-vector space of dimension $k+3$.}}
\end{equation}
We first show that
\begin{equation}\label{eqn:solset}
\text{\parbox{0.9\textwidth}{The solution set of~\eqref{stepkplus1} 
is the solution set of}}
\end{equation}
\begin{equation}\label{eq:theothersolset}
\{e_{m+1},\ldots,e_{m-2k-2},e_{m-2k-3}, d_{m-2k-3}=fc_{m-2k-2},
d_{m-2k-5}=fc_{m-2k-4}\}.
\end{equation}
For ease of reference, we write the equations $e_{m-2k-3}$ and $e_{m-2k-4}$:

\begin{center}
\begin{tabular}{ll}
$e_{m-2k-3}:$ & $c_{m-2k-4}' = -(m-2k-2)fc_{m-2k-2}+d_{m-2k-3}$
\\ $e_{m-2k-4}:$ & $d_{m-2k-5}' = -(m-2k-3)fd_{m-2k-3}+f'c_{m-2k-4}$
\end{tabular}
\end{center}
Suppose $(\tilde{d}_{m-2k-5},\ldots,\tilde{c}_m)$ is a solution of \[\{e_{m+1},\ldots,e_{m-2k-4},d_{m-2k-5}=fc_{m-2k-4}\}.\]
Then $(\tilde{d}_{m-2k-3},\ldots,\tilde{c}_m)$ is a solution of $\{e_{m+1},\ldots,e_{m-2k-2}\}$.  We now show that 
\begin{equation}\label{eq:dm2k3tilde}
\tilde{d}_{m-2k-3}=f\tilde{c}_{m-2k-2}.
\end{equation}
Since $(\tilde{d}_{m-2k-5},\ldots,\tilde{c}_m)$ satisfies $e_{m-2k-4}$, we have
\begin{equation}\label{eqn:lemma2:1}
\tilde{d}_{m-2k-5}'=-(m-2k-3)f\tilde{d}_{m-2k-3} + f'\tilde{c}_{m-2k-4}.
\end{equation}
Since $\tilde{d}_{m-2k-5}=f\tilde{c}_{m-2k-4}$, it follows that
$$
\tilde{d}_{m-2k-5}' = f'\tilde{c}_{m-2k-4} + f\tilde{c}_{m-2k-4}'.
$$
Combining this with \eqref{eqn:lemma2:1}, we get
$$
f\tilde{c}_{m-2k-4}' =-(m-2k-3)f\tilde{d}_{m-2k-3},
$$
and hence
\begin{equation}\label{eqn:lemma2:2}
\tilde{c}_{m-2k-4}' = -(m-2k-3)\tilde{d}_{m-2k-3}.
\end{equation}
Since $(\tilde{d}_{m-2k-5},\ldots,\tilde{c}_m)$ satisfies $e_{m-2k-3}$, we have
$$
\tilde{c}_{m-2k-4}'+(m-2k-2)f\tilde{c}_{m-2k-2}=\tilde{d}_{m-2k-3},
$$
and combining this with \eqref{eqn:lemma2:2} gives us~\eqref{eq:dm2k3tilde}.

We now show the opposite inclusion.  Suppose $(\tilde{d}_{m-2k-5},\ldots,\tilde{c}_m)$ satisfies~\eqref{eq:theothersolset}. 
Since the tuple satisfies $d_{m-2k-5}=fc_{m-2k-4}$, $e_{m-2k-3}$, and $d_{m-2k-3}=fc_{m-2k-2}$, we have
\begin{align*}
\tilde{d}_{m-2k-5}' &= f'\tilde{c}_{m-2k-4} + f\tilde{c}_{m-2k-4}'
\\ &= f'\tilde{c}_{m-2k-4} + f(-(m-2k-2)f\tilde{c}_{m-2k-2}+\tilde{d}_{m-2k-3})
\\ &= f'\tilde{c}_{m-2k-4} + f(-(m-2k-2)\tilde{d}_{m-2k-3} + \tilde{d}_{m-2k-3})
\\ &= f'\tilde{c}_{m-2k-4} - (m-2k-3)f\tilde{d}_{m-2k-3}.
\end{align*}
Thus the tuple also satisfies $e_{m-2k-4}$.  This completes the proof of ~\eqref{eqn:solset}.

Now we show \eqref{eqn:indstep}.
Since~\eqref{stepkplus1} 
is a system consisting of homogeneous linear differential equations and a homogeneous linear equation in $2k+6$ variables, the solution set is a $K$-vector subspace of $K[x]^{2k+6}$.  Let $W$ denote this vector space, let $\pi_i \colon K[x]^{2k+6} \rightarrow K[x]$ be projection onto the $i$-th coordinate, and let $\pi \colon K[x]^{2k+6} \rightarrow K[x]^{2k+4}$ be the projection onto the last $2k+4$ coordinates.  Similarly, the solution set of~\eqref{stepk}
is a $K$-vector subspace of $K[x]^{2k+4}$.  Call this space $V$.  By \eqref{eqn:ih}, $\dim V = k+2$. Let $p_i \colon K[x]^{2k+4} \rightarrow K[x]$ be the projection onto the $i$-th coordinate.

Let $a_1,\ldots,a_{k+2} \in K[x]^{2k+4}$ be a basis for $V$.  For each $i = 1,\ldots,k+2$, we define $b_i \in K[x]^{2k+6}$ as follows.  Let \[\pi(b_i) = a_i,\quad\pi_2(b_i)=\int (-(m-2k-2)fp_2(a_i) + p_1(a_i))dx,\quad \pi_1(b_i)=f\pi_2(b_i).\]  By \eqref{eqn:solset}, each $b_i$ is a solution of~\eqref{stepkplus1}.
Since $d_{m-2k-5}$ and $c_{m-2k-4}$ only appear in the equations 
\begin{gather*}
c_{m-2k-4}'+(m-2k-2)fc_{m-2k-2}=d_{m-2k-3},
\\d_{m-2k-5}'+(m-2k-3)fd_{m-2k-3} = f'c_{m-2k-4},\\
d_{m-2k-5}=fc_{m-2k-4}
\end{gather*}
of~\eqref{stepkplus1}, we observe that \[b_{k+3} := (f,1,0,\ldots,0)\in W.\]  We show that \[\text{span}_K\{b_1,\ldots,b_{k+3}\} = W.\]  Suppose $w\in W$.  By \eqref{eqn:solset}, $\pi(w) \in V$, so there exist $\alpha_i \in K$, $1\leqslant i\leqslant k+2$, such that \[\pi(w) = \sum_{i=1}^{k+2} \alpha_i\pi(b_i).\]  Also by \eqref{eqn:solset}, there is a $\beta \in K$ such that
\begin{align*}
\pi_2(w) &= \int \Big(-(m-2k-2)f\pi_4(w) + \pi_3(w)\Big)dx + \beta
\\ &= \int \Big(-(m-2k-2)f\sum_{i=1}^{k+2}\alpha_i\pi_4(b_i) + \sum_{i=1}^{k+2}\alpha_i\pi_3(b_i)\Big)dx + \beta
\\ &= \sum_{i=1}^{k+2} \alpha_i \int \Big(-(m-2k-2)f\pi_4(b_i) + \pi_3(b_i)\Big)dx + \beta
= \sum_{i=1}^{k+2} \alpha_i\pi_2(b_i) + \beta.
\end{align*}
By \eqref{eqn:solset}, we have $\pi_1(w) = f\pi_2(w)$.  Using the fact that $\pi_1(b_i) = f\pi_2(b_i)$, we get
\begin{equation*}
\pi_1(w) = \sum_{i=1}^{k+2}\alpha_i\pi_1(b_i) + f\beta.
\end{equation*}
Thus, \[w = \sum_{i=1}^{k+2}\alpha_ib_i + \beta b_{k+3}.\]  We conclude that $\text{span}_K\{b_1,\ldots,b_{k+3}\} = W$.

Since $\{\pi(b_1),\ldots,\pi(b_{k+2})\}$ is $K$-linearly independent,  $\{b_1,\ldots,b_{k+2}\}$ is $K$-linearly independent.  Since the constant term of $\pi_2(b_i)$ is $0$ for $i = 1,\ldots,k+2$, it is clear that \[b_{k+3} \not\in \text{span}_K\{b_1,\ldots,b_{k+2}\}.\]  We conclude that $\dim_KW = k+3$.  This completes the inductive step.

Setting $k = \frac{m-3}{2}$ in \eqref{eqn:indstatement} proves the lemma.
\end{proof}

\begin{lemma}\label{lem:02}
If $n \geqslant 3$ is odd, then $\gamma_1 = q\delta$ for some $q \in K[H]$.
\end{lemma}
\begin{proof}
Recall that, if $n \geqslant 3$ is odd, the coefficients of $\gamma_1$ must satisfy $(Io)_n$.  Observe that $\delta(H) = 0$.
Hence, any $K$-derivation $D$ of the form $$D\begin{pmatrix}
x\\y
\end{pmatrix}=\left(a_{\frac{n-1}{2}}H^{\frac{n-1}{2}}+a_{\frac{n-1}{2}-1}H^{\frac{n-1}{2}-1}+\ldots+a_0\right)\cdot\begin{pmatrix}y\\ f\end{pmatrix},\quad a_i\in K,$$
 commutes with $\delta$.  Writing $D$ in the form of \eqref{g as a polynomial}, we see that $c_i = 0$ for even $i$ and $d_i = 0$ for odd $i$, so a choice of $$a_0, \ldots, a_{\frac{n-1}{2}}$$ provides a solution to $(Io)_n$.  Moreover, two distinct choices of $a_0, \ldots, a_{\frac{n-1}{2}}$ provide two distinct solutions of $(Io)_n$.  Thus, the set of solutions of $(Io)_n$ that correspond to derivations of the form $q\delta$, where $q \in K[H]$, is a $K$-vector space of dimension $\frac{n+1}{2}$.  Since this vector space is contained in the vector space of solutions to $(Io)_n$, which by Lemma~\ref{lem:01} has dimension $\frac{n+1}{2}$, the spaces must be equal.
\end{proof}

\begin{lemma}
\label{lem:Ie implies d_m=0}
For all even $m \geqslant 2$, the system $(Ie)_m$ implies $d_m = 0$.
\end{lemma}
\begin{proof}
Fix even $m \geqslant 2$.  Label the equations in $(Ie)_m$ as follows:
\begin{center}
\begin{tabular}{ll}
$e_{m+1}:$ & $d_m' = 0$ \\
$e_m:$ & $c_{m-1}' = d_m$ \\
$e_{m-1}:$ & $d_{m-2}' + mfd_m = f'c_{m-1}$ \\
$e_{m-2}:$ & $c_{m-3}' + (m-1)fc_{m-1} = d_{m-2}$ \\
 & $\vdots$ \\
$e_1:$ & $d_0' + 2fd_2 = f'c_1$ \\
$e_0:$ & $fc_1 = d_0$
\end{tabular}
\end{center}
We show by induction on $k$, $0\leqslant k \leqslant \frac{m-2}{2}$, that 
\begin{equation}\label{eq:implies}
\{e_0,e_1,\ldots,e_{2k+1}\}\quad\text{implies}\quad c_{2k+1}'=-(2k+2)d_{2k+2}.
\end{equation}
The case $k=0$ is straightforward.  For the inductive hypothesis, fix $k$, $0 \leqslant k < \frac{m-2}{2}$, and assume~\eqref{eq:implies}. 
Now assume $\{e_0,e_1,\ldots,e_{2k+3}\}$. 
Equations $e_{2k+2}$ and $e_{2k+3}$ are
$$
c_{2k+1}' = -(2k+3)fc_{2k+3}+d_{2k+2}\quad\text{and}\quad d_{2k+2}'=-(2k+4)fd_{2k+4}+f'c_{2k+3},
$$
and the inductive hypothesis gives us $$c_{2k+1}' = -(2k+2)d_{k+2}.$$  Equating the two expressions for $c_{2k+1}'$, we obtain $d_{2k+2} = fc_{2k+3}$.  Differentiating this and equating the two expressions for $d_{2k+2}'$ gives us
$$
f'c_{2k+3}+fc_{2k+3}' = -(2k+4)fd_{2k+4}+f'c_{2k+3},
$$
which implies $$c_{2k+3}' = -(2k+4)d_{2k+4}.$$  This completes the inductive step.
This shows that a consequence of $(Ie)_m$ is $$c_{m-1}' = -md_m.$$  Since $m$ was assumed to be even, we have $m\neq -1$.  In order that $e_m$ and $c_{m-1}' = -md_m$ both be satisfied, it is necessary that $d_m = 0$.
\end{proof}

\begin{corollary}\label{cor:nnoteven}
It is impossible that $n$ is an even integer greater than or equal to $2$.
\end{corollary}
\begin{proof}
Suppose $n \geqslant 2$ and $n$ is even.  Then the coefficients of $\gamma_1$ must satisfy $(Ie)_n$, and also $d_n \neq 0$.  But by Lemma~\ref{lem:Ie implies d_m=0}, 
$d_n = 0$
is a consequence of $(Ie)_n$.  
\end{proof}

\begin{lemma}
\label{lem:IIe implies c_m=0}
For all even $m \geqslant 2$, the system $(IIe)_m$ implies $c_m = 0$.
\end{lemma}
\begin{proof}
Fix even $m \geqslant 2$.  Label the equations of $(IIe)_m$ as follows:
\begin{center}
\begin{tabular}{ll}
$e_{m+1}:$ & $c_m' = 0$ \\
$e_m:$ & $d_{m-1}' = f'c_m$ \\
$e_{m-1}:$ & $c_{m-2}' + mfc_m = d_{m-1}$ \\
$e_{m-2}:$ & $d_{m-3}' + (m-1)fd_{m-1} = f'c_{m-2}$ \\
 & $\vdots$ \\
$e_1:$ & $c_0' + 2fc_2 = d_1$ \\
$e_0:$ & $fd_1 = f'c_0$
\end{tabular}
\end{center}

We first show the following by induction on $k$, $0\leqslant k \leqslant \frac{m-2}{2}$:
\begin{equation}\label{lemma5ind}
\text{\parbox{.85\textwidth}{If $(\tilde{d}_{m-2k-1},\ldots,\tilde{c}_m)$ is a solution of $\{e_{m+1},\ldots,e_{m-2k}\}$ with $\tilde{c}_m\neq 0$, \newline then $\tilde{d}_{m-2k-1} \neq 0$, $\deg(\tilde{d}_{m-2k-1})=\deg(f\cdot\tilde{c}_{m-2k})$, and $\lc (\tilde{d}_{m-2k-1})=\lc (f\cdot\tilde{c}_{m-2k})$.}}
\end{equation}

\noindent \underline{Base Case, $k=0$:}

Suppose $(\tilde{d}_{m-1},\tilde{c}_m)$ is a solution of $\{c_m' = 0, d_{m-1}' = f'c_m\}$ and $c_m \neq 0$.  
Since $\deg f \geqslant 2$ and $\tilde{c}_m$ is a non-zero constant,  \[\tilde{d}_{m-1} \neq 0\quad \text{and}\quad \deg \tilde{d}_{m-1} = \deg (f\tilde{c}_m) = \deg f.\]  We have $\lc (\tilde{d}_{m-1}') = \deg f \cdot \lc f \cdot \tilde{c}_m$.  Since $\tilde{c}_m$ is a constant and $\deg\tilde{d}_{m-1} = \deg f$, we have \[\lc (\tilde{d}_{m-1}) = \lc (f\tilde{c}_m).\]

\noindent \underline{Inductive Step:}

Fix $k$, $0 \leqslant k < \frac{m-2}{2}$.  Assume~\eqref{lemma5ind} for this $k$. 
Suppose $(\tilde{d}_{m-2k-3},\ldots,\tilde{c}_m)$ is a solution of $\{e_{m+1},\ldots,e_{m-2k-2}\}$ such that $\tilde{c}_m\neq 0$.  For ease of reference, we write:
\begin{center}
\begin{tabular}{ll}
$e_{m-2k-1}:$ & $c_{m-2k-2}' + (m-2k)\cdot f\cdot c_{m-2k} = d_{m-2k-1}$ \\
$e_{m-2k-2}:$ & $d_{m-2k-3}' + (m-2k-1)\cdot f\cdot d_{m-2k-1} = f'\cdot c_{m-2k-2}$
\end{tabular}
\end{center}
Then 
$$
\tilde{c}_{m-2k-2}' = \tilde{d}_{m-2k-1} - (m-2k)f\cdot\tilde{c}_{m-2k}.
$$
Since $m$ is even, 
$m-2k-1\ne 0$.  Therefore, by the inductive hypothesis, 
\begin{equation}\label{eqn:lemma5:1}
\deg(\tilde{c}_{m-2k-2}') = \deg(\tilde{d}_{m-2k-1}) \geqslant 0
\end{equation}
and we have
\begin{equation*}
\lc (\tilde{c}_{m-2k-2}') = -(m-2k-1)\cdot \lc (\tilde{d}_{m-2k-1}),
\end{equation*}
and hence
\begin{equation}\label{eq:forlcdm}
\deg \tilde{c}_{m-2k-2} \cdot \lc (\tilde{c}_{m-2k-2}) = -(m-2k-1)\cdot \lc (\tilde{d}_{m-2k-1}).
\end{equation}
By equation $e_{m-2k-2}$, we have
\begin{equation}\label{eqn:lemma5:2}
\tilde{d}_{m-2k-3}' = f'\cdot\tilde{c}_{m-2k-2} - (m-2k-1)\cdot f\cdot\tilde{d}_{m-2k-1}.
\end{equation}
We will show that the degrees of the two terms on the right-hand side of \eqref{eqn:lemma5:2} are equal and that their leading coefficients do not cancel.  From \eqref{eqn:lemma5:1}, it follows that \[\deg\tilde{c}_{m-2k-2} = \deg\tilde{d}_{m-2k-1} + 1,\] so that
\begin{equation}\label{eqn:lemma5:3}
\deg(f'\cdot\tilde{c}_{m-2k-2}) = \deg(f\cdot\tilde{d}_{m-2k-1}).
\end{equation}
Observe that
$$
\lc (f'\cdot\tilde{c}_{m-2k-2}) = \deg f \cdot \lc f \cdot \lc (\tilde{c}_{m-2k-2})
$$
and, using~\eqref{eq:forlcdm},
$$
\lc (f\cdot\tilde{d}_{m-2k-1}) = \lc f \cdot \lc (\tilde{d}_{m-2k-1}) = \lc f \cdot \tfrac{-1}{m-2k-1}\cdot \lc (\tilde{c}_{m-2k-2})\cdot \deg\tilde{c}_{m-2k-2}.
$$
It follows that 
\begin{equation}\label{eqn:lemma5:4}
\lc (f'\cdot\tilde{c}_{m-2k-2}) \neq (m-2k-1)\cdot\lc (f\cdot\tilde{d}_{m-2k-1}),
\end{equation}
and, together with \eqref{eqn:lemma5:2} and \eqref{eqn:lemma5:3}, this gives us
\begin{equation}\label{eqn:lemma5:5}
\lc (\tilde{d}_{m-2k-3}') = \lc f \cdot \lc (\tilde{c}_{m-2k-2})\cdot(\deg f + \deg\tilde{c}_{m-2k-2}).
\end{equation}
By 
\eqref{eqn:lemma5:2}, \eqref{eqn:lemma5:3}, and \eqref{eqn:lemma5:4}, we have
\begin{equation}\label{eqn:lemma5:6}
\deg(\tilde{d}_{m-2k-3}) = \deg f + \deg\tilde{c}_{m-2k-2}.
\end{equation}
Combining \eqref{eqn:lemma5:5} and \eqref{eqn:lemma5:6} gives us
$$
\lc (\tilde{d}_{m-2k-3}) = \lc f\cdot \lc (\tilde{c}_{m-2k-2}).
$$
This completes the inductive step.

We proceed with the proof of the lemma.  Let $(\tilde{c}_0,\ldots,\tilde{c}_m)$ be a solution of $(IIe)_m$ with $\tilde{c}_m \neq 0$.  We will derive a contradiction.  It follows immediately that $(\tilde{d}_1,\ldots,\tilde{c}_m)$ is a solution of $\{e_{m+1},\ldots,e_1\}$.  Setting $k = \frac{m-2}{2}$ in~\eqref{lemma5ind}, we have that $\deg(\tilde{d}_1) = \deg(f\cdot\tilde{c}_2) \geqslant 0$ and 
\begin{equation}\label{eqn:lc}
\lc (\tilde{d}_1) = \lc (f)\cdot \lc (\tilde{c}_2).
\end{equation}
From $e_0$, we see that \[\deg(\tilde{d}_1) = \deg(\tilde{c}_0) - 1 = \deg(\tilde{c}_0').\]  By equation $e_1$, we have
$$
\lc (\tilde{d}_1) = 2\cdot \lc f \cdot \lc (\tilde{c}_2) + \deg \tilde{c}_0 \cdot \lc (\tilde{c}_0).
$$
Therefore, by \eqref{eqn:lc}, we have
$$
\lc f \cdot \lc (\tilde{c}_2) = 2\cdot \lc f \cdot \lc (\tilde{c}_2) + \deg \tilde{c}_0 \cdot \lc (\tilde{c}_0)
$$
and hence
$$
\lc (\tilde{c}_0) = \frac{-\lc f \cdot \lc (\tilde{c}_2)}{\deg\tilde{c}_0}.
$$
By equation $e_0$, we have
\begin{align*}
\lc f \cdot \lc (\tilde{d}_1) &= \deg f \cdot \lc f \cdot \lc (\tilde{c}_0)
= \deg f \cdot \lc f \cdot \left(\frac{-\lc f \cdot \lc (\tilde{c}_2)}{\deg\tilde{c}_0}\right).
\end{align*}
By \eqref{eqn:lc},
$$
\lc f \cdot \lc f \cdot \lc (\tilde{c}_2) = \deg f \cdot \lc f \cdot \left(\frac{-\lc f \cdot \lc (\tilde{c}_2)}{\deg\tilde{c}_0}\right).
$$
It follows that
$$
\deg\tilde{c}_0 = -\deg f,
$$
which is a contradiction, since $\deg f>0$.
\end{proof}

\begin{corollary}\label{cor:pnoteven}
It is impossible that $p$ is an even integer greater than or equal to $2$.
\end{corollary}
\begin{proof}
Suppose $p \geqslant 2$ and $p$ is even.  Then the coefficients of $\gamma_2$ must satisfy $(IIe)_p$, together with $c_p \neq 0$.  But by Lemma~\ref{lem:IIe implies c_m=0}, $(IIe)_p$ implies $c_p=0$.  
\end{proof}

In the lemmas that follow, we refer to $K$-derivations on the ring $K[x^{1/t},x^{-1/t},y]$, where $t$ is a positive integer.  We view this ring as isomorphic to
$$
K[x,y,z,w] / (z^t - x, zw - 1).
$$
By \cite[Lemma~II.2.1]{Kol}, 
since $\Char K = 0$, any $K$-derivation on $K[x,y]$ extends uniquely to a $K$-derivation on $K[x^{1/t},x^{-1/t},y]$.  One consequence of this is that a $K$-derivation on $K[x^{1/t},x^{-1/t},y]$ can be defined by stating its action on $x$ and $y$.

\begin{lemma}\label{lem:polynomial}
For every odd integer $m$ greater than or equal to $3$, there exists $P_m(X) \in \Z[X]\setminus\{0\}$ such that:
\begin{itemize}
\item $\deg P_m \leqslant \frac{m+1}{2}$
\item for every
\begin{itemize}
\item positive integer $t$
\item $h \in K[x^{1/t},x^{-1/t}]\backslash\{0\}$,
\end{itemize}
if the $K$-derivation 
\[
\beta\begin{pmatrix}x \\ y\end{pmatrix} = \begin{pmatrix}c_{m-1}y^{m-1}+c_{m-3}y^{m-3}+\ldots+c_0 \\ d_my^m+d_{m-2}y^{m-2}+\ldots+d_1y\end{pmatrix}
\] on $K[x^{1/t},x^{-1/t},y]$
commutes with the $K$-derivation $$
\alpha\begin{pmatrix} x \\ y\end{pmatrix} = \begin{pmatrix}y \\ h\end{pmatrix}
$$ on $K[x^{1/t},x^{-1/t},y]$,   then \[P_m(N) = 0\quad \text{or}\quad N \in \{-1\} \cup \left\{-\tfrac{k}{k-1}\:\Big|\: 2\leqslant k \leqslant \tfrac{m+1}{2}\right\},\]
where $N = \deg h$, each $c_i, d_i \in K[x^{1/t},x^{-1/t}]$ and $d_m \neq 0$.
\end{itemize}
\end{lemma}

\begin{proof}
Fix $m \geqslant 3$.  For $i=0,\ldots,m$, we define $T_i(X) \in \Z [X]$ as follows.  Let \[T_m(X) = T_{m-1}(X) = 1.\]  For $1 \leqslant k \leqslant \frac{m-1}{2}$, let
\begin{equation}\label{eqn:lemma6:defta}
T_{m-2k}(X) = X \cdot T_{m-(2k-1)}(X) - (m-(2k-2)) \cdot ((k-1)\cdot(X+1)+1) \cdot T_{m-(2k-2)}(X)
\end{equation}
and let
\begin{equation}\label{eqn:lemma6:deftb}
T_{m-(2k+1)}(X) = T_{m-2k}(X) - (m-(2k-1)) \cdot k\cdot(X+1) \cdot T_{m-(2k-1)}(X).
\end{equation}
Let
\begin{equation}\label{eqn:lemma6:defpm}
P_m(X) = \left(\tfrac{m-1}{2} \cdot (X+1) + 1\right) \cdot T_1(X) - X \cdot T_0(X).
\end{equation}

We first prove that 
\begin{equation}\label{eq:degPm}
\deg P_m(X) \leqslant \tfrac{m+1}{2}.
\end{equation}
We show by induction on $k$, $0 \leqslant k \leqslant \frac{m-1}{2}$, that
\begin{equation}\label{eq:Tm}
\deg T_{m-2k}(X) \leqslant k\ \text{ and }\ \deg T_{m-(2k+1)}(X) \leqslant k.
\end{equation}
For the base case, $k=0$, we have \[\deg T_m(X) = \deg T_{m-1}(X) = 0.\]  For the inductive step, fix $k$, $0 \leqslant k < \frac{m-1}{2}$, and assume~\eqref{eq:Tm}.  
It follows from \eqref{eqn:lemma6:defta} and the inductive hypothesis that \[\deg T_{m-(2k+2)}(X) \leqslant k+1,\] and it follows from \eqref{eqn:lemma6:deftb} and the inductive hypothesis that \[\deg T_{m-(2k+3)}(X) \leqslant k+1.\]  This completes the proof by induction.  As a consequence, we have \[\deg T_1(X) \leqslant \tfrac{m-1}{2}\quad \text{and}\quad \deg T_0(X) \leqslant \tfrac{m-1}{2}.\]  Therefore,~\eqref{eq:degPm} holds.
Next, we show that $P_m(X)$ is not the zero polynomial.  To this end, we first prove by induction on $k$, $0 \leqslant k \leqslant \frac{m-1}{2}$, that 
\begin{equation}\label{eqn:lemma6:tnotzero}
T_{m-2k}(-1) \neq 0\quad \text{and}\quad T_{m-(2k+1)}(-1) \neq 0.
\end{equation}
The base case, $k=0$, is trivial, since $T_m(X) = T_{m-1}(X) = 1$.  For the inductive hypothesis, fix $k$, $0 \leqslant k < \frac{m-1}{2}$, and assume \[T_{m-2k}(-1) \cdot T_{m-(2k+1)}(-1) \neq 0.\]  Equation \eqref{eqn:lemma6:deftb} shows that 
$$
T_{m-(2k+1)}(-1) = T_{m-2k}(-1).
$$
Replacing $k$ with $k+1$ in \eqref{eqn:lemma6:defta} gives us
\begin{align*}
T_{m-(2k+2)}(-1) &= -1 \cdot T_{m-(2k+1)}(-1) - (m-2k) \cdot T_{m-2k}(-1)
= -(m-2k+1) \cdot T_{m-2k}(-1).
\end{align*}
Since $k < \frac{m-1}{2}$, it must be that $m-2k+1 \neq 0$.  Now by the inductive hypothesis, \[T_{m-(2k+2)}(-1) \neq 0.\]  Replacing $k$ with $k+1$ in \eqref{eqn:lemma6:deftb} yields
$$
T_{m-(2k+3)}(-1) = T_{m-(2k+2)}(-1) \neq 0.
$$
This completes the proof of \eqref{eqn:lemma6:tnotzero}.
By \eqref{eqn:lemma6:defpm}, we have
$$
P_m(-1) = T_1(-1) + T_0(-1).
$$
Replacing $k$ with $\frac{m-1}{2}$ in \eqref{eqn:lemma6:deftb} gives
$$
T_0(-1) = T_1(-1),
$$
and hence
$$
P_m(-1) = 2 \cdot T_1(-1) \neq 0.
$$
This completes the proof that $P_m(X)$ is not the zero polynomial.

We proceed to show that $P_m(X)$ satisfies the remaining property stated in the lemma.  Fix $t \in \mathbb{Z}^{\geq 1}$, fix $h \in K[x^{1/t},x^{-1/t}]\backslash\{0\}$, and define $\alpha$ as in the statement of the lemma.  Fix $\beta$ as in the statement of the lemma.  Note that $c_i$ and $d_i$ must satisfy the equations of system $(IIo)_m$, with $f$ replaced by $h$.  Label these equations as follows:
\begin{center}
\begin{tabular}{ll}
$e_{m+1}:$ & $d_m'= 0$ \\
$e_m:$ & $c_{m-1}'= d_m$ \\
$e_{m-1}:$ & $d_{m-2}'+ mhd_m = h'c_{m-1}$ \\
$e_{m-2}:$ & $c_{m-3}' + (m-1)hc_{m-1} = d_{m-2}$ \\
\vdots & \vdots \\
$e_{m-(2k-1)}:$ & $d_{m-2k}' + (m-(2k-2))hd_{m-(2k-2)} = h'c_{m-(2k-1)}$ \\
$e_{m-2k}:$ & $c_{m-(2k+1)}' + (m-(2k-1))hc_{m-(2k-1)} = d_{m-2k}$ \\
$e_{m-(2k+1)}:$ & $d_{m-(2k+2)}' + (m-2k)hd_{m-2k} = h'c_{m-(2k+1)}$ \\
$e_{m-(2k+2)}:$ & $c_{m-(2k+3)}' + (m-(2k+1))hc_{m-(2k+1)} = d_{m-(2k+2)}$ \\
\vdots & \vdots \\
$e_0:$ & $hd_1 = h'c_0$
\end{tabular}
\end{center}
Let $N = \deg h$ and let $L = \lc (h)$.  Assume that
$$
N \not\in \{-1\} \cup \left\{-\tfrac{k}{k-1}\:\Big|\: 2\leqslant k \leqslant \tfrac{m+1}{2}\right\}.
$$
We first show by induction that for all $k$, $0 \leqslant k \leqslant \frac{m-1}{2}$,
\begin{equation}\label{eqn:lemma6:deg}
\deg d_{m-2k} \leqslant k(N+1) \quad \text{ and } \quad \deg c_{m-(2k+1)} \leqslant k(N+1)+1.
\end{equation}
We first treat the base case, 
$k=0$.  By equations $e_{m+1}$ and $e_m$,  $\deg d_m \leqslant 0$ and $\deg c_{m-1} \leqslant 1$.

For the inductive hypothesis, fix $k$, $0 \leqslant k < \frac{m-1}{2}$ and assume~\eqref{eqn:lemma6:deg}.
Consider $e_{m-(2k+1)}$.  By the inductive hypothesis, we have \[\deg (hd_{m-2k}) \leqslant k(N+1)+N\quad \text{and}\quad \deg (h'c_{m-(2k+1)}) \leqslant k(N+1)+N.\]  It follows that
\begin{equation}\label{eqn:lemma6:deg:1}
\deg d_{m-(2k+2)} \leqslant (k+1)(N+1).
\end{equation}
Now consider $e_{m-(2k+2)}$.  By the inductive hypothesis, \[\deg (hc_{m-(2k+1)}) \leqslant (k+1)(N+1).\]  It follows from this and \eqref{eqn:lemma6:deg:1} that
$$
\deg c_{m-(2k+3)} \leqslant (k+1)(N+1)+1.
$$
This concludes the proof of \eqref{eqn:lemma6:deg} for all $k$, $0 \leqslant k \leqslant \frac{m-1}{2}$.

Define $a_m, a_{m-1},\ldots,a_0$ as follows.  Let 
\begin{gather*}a_{m-2k} = \text{ the coefficient of } x^ {k(N+1)} \text{ in } d_{m-2k},\\
a_{m-(2k+1)} = \text{ the coefficient of } x^{k(N+1)+1} \text{ in } c_{m-(2k+1)}.
\end{gather*}
Equations $e_{m+1}$ and $e_m$ and the requirement that $d_m \neq 0$ imply that $a_{m-1} = a_m$.  Now we prove that, for all $k$, $1 \leqslant k \leqslant \frac{m-1}{2}$,
\begin{equation}\label{eqn:lemma6:prelim1}
a_{m-(2k+1)}=\left(a_{m-2k}-(m-(2k-1))\cdot L\cdot a_{m-(2k-1)}\right)\cdot\tfrac{1}{k(N+1)+1}
\end{equation}
and
\begin{equation}\label{eqn:lemma6:prelim2}
a_{m-2k} = (L\cdot N\cdot a_{m-(2k-1)}-(m-(2k-2))\cdot L\cdot a_{m-(2k-2)})\cdot\tfrac{1}{k(N+1)}.
\end{equation}
Fix $k$, $1 \leqslant k \leqslant \frac{m-1}{2}$.  By equation $e_{m-(2k-1)}$, we have
\begin{equation}\label{eqn:lemma6:2}
d_{m-2k}'=h'c_{m-(2k-1)}-(m-(2k-2))\cdot h\cdot d_{m-(2k-2)}.
\end{equation}
Let us write an equation equating the coefficients of $x^{k(N+1)-1}$ on both sides of \eqref{eqn:lemma6:2}. First, observe that the coefficient of $x^{k(N+1)-1}$ in $d_{m-2k}'$ is $k(N+1) \cdot a_{m-2k}$.  Next consider $h'c_{m-(2k-1)}$.  First consider the case $N \neq 0$.  It follows that $\deg h' = N - 1$.  By \eqref{eqn:lemma6:deg}, we have \begin{equation}
\label{eq:degcm2k1}
\deg c_{m-(2k-1)} \leqslant k(N+1)-N.
\end{equation}
Thus, the coefficient of $x^{k(N+1)-1}$ in $h'c_{m-(2k-1)}$ is $N \cdot L \cdot a_{m-(2k-1)}$.  Now consider the case $N=0$. Either $h'=0$, or $h'\neq 0$ and $\deg h' < N-1$.  If $h'=0$, then $h'c_{m-(2k-1)} = 0$ and the coefficient of $x^{k(N+1)-1}$ in $h'c_{m-(2k-1)}$ is $0$, which is equal to $L \cdot N \cdot a_{m-(2k-1)}$.  If $N=0$ and $h' \neq 0$, then, since $\deg h' < N-1$ and by~\eqref{eq:degcm2k1},  the coefficient of $x^{k(N+1)-1}$ in $h'c_{m-(2k-1)}$ is $0$, which is equal to $L \cdot N \cdot a_{m-(2k-1)}$.  Finally, consider $h \cdot d_{m-(2k-2)}$.  Since $\deg h = N$ and, by \eqref{eqn:lemma6:deg}, \[\deg d_{m-(2k-2)} \leqslant k(N+1)-N-1,\] we see that the coefficient of $x^{k(N+1)-1}$ in $h \cdot d_{m-(2k-2)}$ is $L \cdot a_{m-(2k-2)}$.  Since $N\neq -1$, we have $k(N+1) \neq 0$.  Thus, equating the coefficients of $x^{k(N+1)-1}$ in \eqref{eqn:lemma6:prelim2} yields 
\begin{equation*}\label{eqn:ind1}a_{m-2k} = (L\cdot N\cdot a_{m-(2k-1)}-(m-(2k-2))\cdot L\cdot a_{m-(2k-2)})\cdot\tfrac{1}{k(N+1)}.\end{equation*}
By equation $e_{m-2k}$, we have
\begin{equation}\label{eqn:lemma6:3}
c_{m-(2k+1)}'=d_{m-2k}-(m-(2k-1))\cdot h\cdot c_{m-(2k-1)}.
\end{equation}
Let us write an equation equating the coefficients of $x^{k(N+1)}$ on either side of \eqref{eqn:lemma6:3}.  The coefficient of $x^{k(N+1)}$ in $c_{m-(2k+1)}'$ is $(k(N+1)+1) \cdot a_{m-(2k+1)}$.  The coefficient of $x^{k(N+1)}$ in $d_{m-2k}$ is $a_{m-2k}$.  By \eqref{eqn:lemma6:deg}, we have \[\deg c_{m-(2k-1)} \leqslant k(N+1)-N,\] and, since $\deg h = N$, the coefficient of $x^{k(N+1)}$ in $hc_{m-(2k-1)}$ is $L \cdot a_{m-(2k-1)}$.  Since $N\neq -\frac{k+1}{k}$, we have $k(N+1)+1 \neq 0$.  Thus, equating the coefficients of $x^{k(N+1)}$ on either side of \eqref{eqn:lemma6:3} yields
\begin{equation*}\label{eqn:ind2}a_{m-(2k+1)}=\left(a_{m-2k}-(m-(2k-1))\cdot L\cdot a_{m-(2k-1)}\right)\cdot\tfrac{1}{k(N+1)+1}.\end{equation*}
This concludes the proof of \eqref{eqn:lemma6:prelim1} and \eqref{eqn:lemma6:prelim2}.

For $i = 0, \ldots, m$, define $S_i \in \Z$ as follows.  Let \[S_m = S_{m-1} = 1.\]  For every $k$, $1 \leqslant k \leqslant \frac{m-1}{2}$, let \[S_{m-2k} = k(N+1)\cdot S_{m-(2k-1)}\quad\text{and}\quad S_{m-(2k+1)} = (k(N+1)+1)\cdot S_{m-2k}.\]

Next, we prove by induction that for all $k$, $0 \leqslant k \leqslant \frac{m-1}{2}$, we have
\begin{equation}\label{eqn:lemma6:tn1}
T_{m-2k}(N) = S_{m-2k}\cdot \tfrac{1}{L^k} \cdot \tfrac{1}{a_m} \cdot a_{m-2k}\ \text{ and }\  T_{m-(2k+1)}(N) = S_{m-(2k+1)} \cdot \tfrac{1}{L^k} \cdot \tfrac{1}{a_m} \cdot a_{m-(2k+1)}.
\end{equation}
Recall that by our assumption on the form of $\beta$, we have $a_m \neq 0$.

The base case, $k=0$, is proved immediately by noting that $a_m = a_{m-1}$ follows from $e_{m+1}$ and $e_m$.

For the inductive hypothesis, fix $k$, $0 \leqslant k < \frac{m-1}{2}$ and assume \eqref{eqn:lemma6:tn1} holds.  We have from \eqref{eqn:lemma6:defta}, \eqref{eqn:lemma6:tn1}, and the definition of $S_i$ that
\begin{align*}
T_{m-(2k+2)}(N) &= N \cdot T_{m-(2k+1)}(N) - (m-2k) (k(N+1)+1)\cdot T_{m-2k}(N)
\\ &= N \cdot S_{m-(2k+1)} \cdot \tfrac{1}{L^k} \cdot \tfrac{1}{a_m} \cdot a_{m-(2k+1)} - (m-2k)(k(N+1)+1) \cdot S_{m-2k} \cdot \tfrac{1}{L^k} \cdot \tfrac{1}{a_m} \cdot a_{m-2k}
\\ &= N \cdot \tfrac{S_{m-(2k+2)}}{(k+1)(N+1)} \cdot \tfrac{1}{L^k} \cdot \tfrac{1}{a_m} \cdot a_{m-(2k+1)} - (m-2k) \cdot \tfrac{S_{m-(2k+2)}}{(k+1)(N+1)} \cdot \tfrac{1}{L^k} \cdot \tfrac{1}{a_m} \cdot a_{m-2k}
\\ &= S_{m-(2k+2)} \cdot \tfrac{1}{L^k} \cdot \tfrac{1}{a_m} \cdot \left(Na_{m-(2k+1)} - (m-2k)a_{m-2k}\right) \cdot \tfrac{1}{(k+1)(N+1)}.
\end{align*}
From \eqref{eqn:lemma6:prelim2} with $k$ replaced by $k+1$, we see that
\begin{equation}\label{eqn:lemma6:ind1part1}
T_{m-(2k+2)}(N) = S_{m-(2k+2)} \cdot \tfrac{1}{L^{k+1}}\cdot \tfrac{1}{a_m} \cdot a_{m-(2k+2)}.
\end{equation}
We have from \eqref{eqn:lemma6:deftb}, \eqref{eqn:lemma6:ind1part1}, \eqref{eqn:lemma6:tn1}, and the definition of $S_i$ that
\begin{align*}
&T_{m-(2k+3)}(N) = T_{m-(2k+2)}(N) - (m-(2k+1))\cdot (k+1)(N+1) \cdot T_{m-(2k+1)}(N)
\\& = S_{m-(2k+2)} \cdot \tfrac{1}{L^{k+1}} \cdot \tfrac{1}{a_m} \cdot a_{m-(2k+2)} - (m-(2k+1)) \cdot (k+1)(N+1) \cdot S_{m-(2k+1)} \cdot \tfrac{1}{L^k} \cdot \tfrac{1}{a_m} \cdot a_{m-(2k+1)}
\\ &= S_{m-(2k+3)} \cdot \tfrac{1}{L^{k+1}} \cdot \tfrac{1}{a_m} \cdot \left( a_{m-(2k+2)} - L(m-(2k+1))a_{m-(2k+1)} \right) \cdot \tfrac{1}{(k+1)(N+1) + 1}.
\end{align*}
From \eqref{eqn:lemma6:prelim1} with $k$ replaced by $k+1$, we see that
$$
T_{m-(2k+3)}(N) = S_{m-(2k+3)} \cdot \tfrac{1}{L^{k+1}} \cdot \tfrac{1}{a_m} \cdot a_{m-(2k+3)}.
$$
This completes the proof of \eqref{eqn:lemma6:tn1}.

Now we show that $P_m(N) = 0$.  Using $k = \frac{m-1}{2}$ in \eqref{eqn:lemma6:tn1} and  $S_0 = (\tfrac{m-1}{2}(N+1)+1)S_1$, we have
\begin{align*}
P_m(N) &= \left(\tfrac{m-1}{2} (N+1) + 1\right) \cdot T_1(N) - N \cdot T_0(N)
\\ &= \left(\tfrac{m-1}{2} (N+1) + 1\right) S_1 \cdot \tfrac{1}{L^{(m-1)/2}} \cdot \tfrac{1}{a_m} \cdot a_1 - N \cdot S_0 \cdot \tfrac{1}{L^{(m-1)/2}} \cdot \tfrac{1}{a_m} \cdot a_0 = \tfrac{S_0}{L^{(m-1)/2}} \cdot \tfrac{1}{a_m} \cdot \left( a_1 - Na_0 \right).
\end{align*}
Consider equation $e_0$:
$$
hd_1 = h'c_0.
$$
Equating the coefficients of $x^{(N+1)((m-1)/2)+N}$ in $e_0$, recalling \eqref{eqn:lemma6:deg}, yields
$$
a_1 = Na_0.
$$
We conclude that $P_m(N) = 0$.
\end{proof}

\begin{lemma}\label{lem:positive}
For every  positive integer $k$,
the $K$-derivation $$\alpha\begin{pmatrix}x \\ y\end{pmatrix} = \begin{pmatrix}y \\ x^{-\frac{2k+1}{2k-1}}\end{pmatrix}$$ of the ring $K[x^{-\frac{1}{2k-1}},x^{\frac{1}{2k-1}},y]$ commutes with the $K$-derivation $$\beta\begin{pmatrix}x \\ y\end{pmatrix} = \begin{pmatrix}\sum_{l=0}^ka_{2(k-l)}x^{1+(1-\frac{2k+1}{2k-1})l}y^{2(k-l)} \\ \sum_{l=0}^k a_{2(k-l)+1}x^{(1-\frac{2k+1}{2k-1})l}y^{2(k-l)+1}\end{pmatrix},$$ where the $a_i \in K$ are defined recursively as follows:  $a_{2k+1} \in K\setminus \{0\}$ is arbitrary, $a_{2k} = a_{2k+1}$, and for $0<l\leqslant k$,
\begin{gather*}
a_{2(k-l)+1} = \left(-\tfrac{2k+1}{2k-1}a_{2(k-l)+2}-(2(k-l)+3)a_{2(k-l)+3}\right)\left((1-\tfrac{2k+1}{2k-1})l\right)^{-1}\\
a_{2(k-l)} = \left(a_{2(k-l)+1}-(2(k-l)+2)a_{2(k-l)+2}\right)\left((1-\tfrac{2k+1}{2k-1})l+1\right)^{-1}.
\end{gather*}
\end{lemma}

\begin{proof}

We first show that \[\beta(\alpha(x)) = \alpha(\beta(x)).\]  We have $\beta(\alpha(x)) = \beta(y)$.  Note that, in $\beta(y)$, only odd powers of $y$ with exponents less than or equal to $2k+1$ appear, and for all $l$, $0 \leqslant l \leqslant k$, the coefficient of $y^{2(k-l)+1}$ is
\begin{equation}\label{eq:betayl}
a_{2(k-l)+1}x^{(1-\frac{2k+1}{2k-1})l}.
\end{equation}
In $\alpha(\beta(x))$, only odd powers of $y$ with exponents less than or equal to $2k+1$ appear.  The coefficient of $y^{2k+1}$ is $a_{2k}$, which equals $a_{2k+1}$, which is the coefficient of $y^{2k+1}$ in $\beta(\alpha(x))$.  For all $l$, $1 \leqslant l \leqslant k$, the coefficient of $y^{2(k-l)+1}$ in $\alpha(\beta(x))$ is 
$$
a_{2(k-l)}x^{\left(1-\frac{2k+1}{2k-1}\right)l}\left(1+(1-\tfrac{2k+1}{2k-1})l\right) + a_{2(k-l)+2}x^{\left(1-\frac{2k+1}{2k-1}\right)l}(2(k-l)+2).
$$  
By the definition of $a_{2(k-l)}$, this equals~\eqref{eq:betayl}.  
Now we show that \[\beta(\alpha(y)) = \alpha(\beta(y)).\]  We have $$\beta(\alpha(y)) = \beta\left(x^{-\frac{2k+1}{2k-1}}\right) = -\tfrac{2k+1}{2k-1}x^{-\frac{2k+1}{2k-1}-1}\beta(x).$$  This expression contains only even powers of $y$ from $y^0$ to $y^{2k}$.  For all $l$, $0 \leqslant l \leqslant k$, the coefficient of $y^{2(k-l)}$ in $\beta(\alpha(y))$ is 
\begin{equation}\label{eq:y2kl}
-\tfrac{2k+1}{2k-1}a_{2(k-l)}x^{\left(1-\frac{2k+1}{2k-1}\right)l-\frac{2k+1}{2k-1}}.
\end{equation}
We see that $\alpha(\beta(y))$ contains only even powers of $y$ from $y^0$ to $y^{2k}$.  For $l<k$, the coefficient of $y^{2(k-l)}$ in $\alpha(\beta(y))$ is
$$
a_{2(k-l)+1}x^{\left(1-\frac{2k+1}{2k-1}\right)l-\frac{2k+1}{2k-1}}(2(k-l)+1) + a_{2(k-l)-1}x^{\left(1-\frac{2k+1}{2k-1}\right)l - \frac{2k+1}{2k-1}}\left(1-\tfrac{2k+1}{2k-1}\right)(l+1).
$$
By definition,
$$
a_{2(k-l)-1} = \left(-\tfrac{2k+1}{2k-1}a_{2(k-l)} - (2(k-l)+1)a_{2(k-l)+1}\right)\left(\left(1-\tfrac{2k+1}{2k-1}\right)(l+1)\right)^{-1}.
$$
Hence, the coefficient of $y^{2(k-l)}$ in $\alpha(\beta(y))$ is~\eqref{eq:y2kl}.  
The coefficient of $y^0$ in $\alpha(\beta(y))$ is \[a_1x^{\left(1-\frac{2k+1}{2k-1}\right)k-\frac{2k+1}{2k-1}}.\]  It remains to show that 
\begin{equation}\label{eq:a1a0}
a_1 = -\tfrac{2k+1}{2k-1}a_0.
\end{equation}
This is an immediate consequence of the following lemma.

\begin{lemma}\label{induction} In the notation of Lemma~\ref{lem:positive},
for all $l$, $0 \leqslant l \leqslant k$, $$\tfrac{2k+1}{2k-1}a_{2(k-l)} = \tfrac{2(k-l)+1}{2(k-l)-1}a_{2(k-l)+1}.$$
\end{lemma}
\begin{proof}
We proceed by induction on $l$.  The base case $l = 0$ is immediate, since by definition $a_{2k} = a_{2k+1}$.
For the inductive hypothesis, fix $l < k$ and assume $$\tfrac{2k+1}{2k-1}a_{2(k-l)} = \tfrac{2(k-l)+1}{2(k-l)-1}a_{2(k-l)+1}.$$  We want to show that \begin{equation}\label{wts}\tfrac{2k+1}{2k-1}a_{2(k-l)-2} = \tfrac{2(k-l)-1}{2(k-l)-3}a_{2(k-l)-1}.\end{equation}
The left-hand side of \eqref{wts} is, by the definition of $a_{2(k-l)-2}$,
$$
\tfrac{2k+1}{2k-1}\cdot\frac{\left(a_{2(k-l)-1}-2(k-l)a_{2(k-l)}\right)}{\left(1-\tfrac{2k+1}{2k-1}\right)(l+1)+1}.
$$
By the definition of $a_{2(k-l)-1}$, this equals
$$
\tfrac{2k+1}{2k-1}\cdot\frac{\left(\frac{-\tfrac{2k+1}{2k-1}a_{2(k-l)}-(2(k-l)+1)a_{2(k-l)+1}}{\left(1-\tfrac{2k+1}{2k-1}\right)(l+1)}-2(k-l)a_{2(k-l)}\right)}{\left(1-\tfrac{2k+1}{2k-1}\right)(l+1)+1}.
$$
By the inductive hypothesis, this is equal to
\begin{equation}\label{LHS}
\frac{\left(\frac{-\tfrac{2k+1}{2k-1}\cdot\tfrac{2(k-l)+1}{2(k-l)-1}-\tfrac{2k+1}{2k-1}(2(k-l)+1)}{\left(1-\tfrac{2k+1}{2k-1}\right)(l+1)}-2(k-l)\tfrac{2(k-l)+1}{2(k-l)-1}\right)}{\left(1-\tfrac{2k+1}{2k-1}\right)(l+1)+1}\cdot a_{2(k-l)+1}.
\end{equation}

The right-hand side of \eqref{wts} is, using the definition of $a_{2(k-l)-1}$,
$$
\tfrac{2(k-l)-1}{2(k-l)-3}\cdot\frac{-\tfrac{2k+1}{2k-1}a_{2(k-l)}-(2(k-l)+1)a_{2(k-l)+1}}{\left(1-\tfrac{2k+1}{2k-1}\right)(l+1)}.
$$
By the inductive hypothesis, this equals
$$
\tfrac{2(k-l)-1}{2(k-l)-3}\frac{-\tfrac{2(k-l)+1}{2(k-l)-1} - (2(k-l)+1)}{\left(1-\frac{2k+1}{2k-1}\right)(l+1)}\cdot a_{2(k-l)+1},
$$
which is equal to \eqref{LHS}, as a computation shows.
\end{proof}

By letting $l = k$ in Lemma~\ref{induction}, we see that~\eqref{eq:a1a0} holds.
\end{proof}

\begin{lemma}\label{lem:07}
For every
\begin{itemize}
\item positive integer $t$, 
\item $h \in K[x^{1/t}, x^{-1/t}] \backslash \{0\}$,
\end{itemize}
if there exists a $K$-derivation $\beta$ on $K[x^{1/t},x^{-1/t},y]$ such that
\begin{itemize}
\item $\beta$ commutes with the $K$-derivation 
$$\alpha\begin{pmatrix}x \\ y\end{pmatrix} = \begin{pmatrix}y \\ h\end{pmatrix}$$ on $K[x^{-1/t},x^{1/t},y]$ and 
\item $\beta$ is of the form 
$$
\beta\begin{pmatrix}x \\ y\end{pmatrix} = \begin{pmatrix}c_{m-1}y^{m-1}+c_{m-3}y^{m-3}+\ldots+c_0 \\ d_my^m+d_{m-2}y^{m-2}+\ldots+d_1y\end{pmatrix},
$$
where $m \geqslant 3$ is odd, $c_i$, $d_i \in K[x^{-1/t},x^{1/t}]$, and $d_m \neq 0$,
\end{itemize}
then \[N := \deg h \in S\cup T,\] where
\begin{gather*}S = \{1\} \cup \left\{-\tfrac{2k+1}{2k-1} \: \big| \: k \in \Z, 1\leqslant k \leqslant \tfrac{m-1}{2}\right\},\\T = \{-1\} \cup \left\{-\tfrac{k}{k-1} \: \big| \: k \in \Z, k \geqslant 2\right\}.
\end{gather*}
\end{lemma}

\begin{proof}
Fix $t \in \Z^{\geq 1}$.  Fix $h \in K[x^{-1/t},x^{1/t}]\backslash \{0\}$ and hence $\alpha$ of the form stated in the lemma.  Let $N = \deg h$ and assume $N \not\in T$.  Suppose a $K$-derivation $\beta$ satisfying the properties stated in the lemma exists and let $m$ be the least odd integer greater than or equal to $3$ such that there exists such a $\beta$.  
By Lemma~\ref{lem:polynomial}, $P_m(N)=0$, and $P_m$ has at most $\frac{m+1}{2}$ zeros.  We show that these zeros are exactly the elements of $S$.

We show that $P_m(1) = 0$.  The $K$-derivations 
$$
\partial_1\begin{pmatrix}x \\ y\end{pmatrix} = \begin{pmatrix}y \\ x\end{pmatrix} \quad \text{ and } \quad \partial_2\begin{pmatrix}x \\ y\end{pmatrix} = \begin{pmatrix}x \\ y\end{pmatrix}
$$
on $K[x,x^{-1},y]$ commute and $\partial_1$ has the form of $\alpha$ in the statement of Lemma~\ref{lem:polynomial}.  The polynomial $r := y^2-x^2$ is a first integral of $\partial_1$, and so $r^{(m-1)/2}\partial_2$ is a $K$-derivation commuting with $\partial_1$ of the form of $\beta$ in the statement of Lemma~\ref{lem:polynomial}.  Therefore, by Lemma~\ref{lem:polynomial}, $P_m(1) = 0$.

We show that, for all $k$, $1 \leqslant k \leqslant \frac{m-1}{2}$,
\begin{equation}\label{eq:zeroofPm}
P_m\big(-\tfrac{2k+1}{2k-1}\big) = 0.
\end{equation}Fix $k$.  Let $K$-derivations $\partial_1$ and $\partial_2$ on $K[x^{\frac{1}{2k-1}},x^{-\frac{1}{2k-1}},y]$ be defined as $\alpha$ and $\beta$ are in Lemma~\ref{lem:positive}.  Now \[r = y^2 + 2\left(\tfrac{2k-1}{2}\right)x^{-2/(2k-1)}\] is a first integral of $\partial_1$.  Note that $\deg_y\partial_2(y) = 2k+1$.  Now $r^{(m-(2k+1))/2}\partial_2$ is a derivation commuting with $\partial_1$ of the form of $\beta$ of Lemma~\ref{lem:polynomial}.  Hence, we have~\eqref{eq:zeroofPm}.

The set $S$ consists of $\frac{m+1}{2}$ elements, and we have shown that each is a zero of $P_m$, which is nonzero of degree at most $\frac{m+1}{2}$.  It follows that $S$ is exactly the zero set of $P_m$.
\end{proof}

\begin{corollary}\label{cor:pnotodd}
It is impossible that $p$ is an odd integer greater than or equal to $3$.
\end{corollary}
\begin{proof}
Suppose $p \geqslant 3$ and $p$ is odd.  Let $N = \deg f$.  Recall that $p = \deg_y\gamma_2$.  Consider Lemma~\ref{lem:07}.  Since the extensions of $\delta$ and $\gamma_2$ to $K$-derivations on $K[x,x^{-1},y]$ are of the forms of $\alpha$ and $\beta$, it follows that $N \in S \cup T$.  Since $N$ is assumed to be an integer greater than or equal to $2$, this is a contradiction.
\end{proof}

This finishes the proof of Theorem~\ref{thm:main}.
\end{proof}

\section*{Acknowledgements} This work was partially supported by the NSF grants CCF-0952591, DMS-1700336,
 DMS-1606334,
by the NSA grant \#H98230-15-1-0245, by
CUNY CIRG \#2248, and by PSC-CUNY grants \#69827-00 47 and \#60456-00 48.
The authors are grateful to the CCiS at CUNY Queens College for the computational resources.
\bibliographystyle{abbrvnat}
\bibliography{bibdata}
\end{document}